\documentclass[preprint, 11pt,nonatbib]{elsarticle}
\usepackage[utf8]{inputenc}
\usepackage{todonotes}
\usepackage{amssymb}
\usepackage{amsmath}
\usepackage{amsthm}
\usepackage{tikz-cd}
\usepackage{verbatim}

\usepackage[title]{appendix}

\makeatletter
\let\c@author\relax
\makeatother

\usepackage[
backend=biber,
sorting=nty,doi=false,url=false
]{biblatex}
\DeclareMathOperator{\Dom}{Val}

\DeclareMathOperator{\Spec}{Spec}
\DeclareMathOperator{\Zar}{Zar}

\newtheorem{theorem}{Theorem}
\newtheorem{proposition}[theorem]{Proposition}
\newtheorem{corollary}[theorem]{Corollary}
\newtheorem{lemma}[theorem]{Lemma}

\theoremstyle{definition}

\newtheorem{notation}[theorem]{Notation}

\newtheorem{example}[theorem]{Example}

\numberwithin{theorem}{section}

\begin{document}

\begin{frontmatter}



\title{A connectedness theorem for spaces of valuation rings} 

%\title{Integrally closed local rings and connected sets of valuation rings}

\author[Heinzer]{William Heinzer}
\address[Heinzer]{Department of Mathematics, Purdue University, West
Lafayette, Indiana 47907-1395}
\ead{heinzer@purdue.edu}

\author[Loper]{K.~Alan Loper}
\address[Loper]{Department of Mathematics, Ohio State University -- Newark,  Newark, OH 43055}
\ead{lopera@math.ohio-state.edu}

\author[Olberding]{Bruce Olberding} 
\address[Olberding]{Department of Mathematical Sciences, New Mexico State University,
 Las Cruces, NM 88003-8001}
 \ead{bruce@nmsu.edu}

\author[Toeniskoetter]{Matthew Toeniskoetter} 
\address[Toeniskoetter]{Department of Mathematics and Statistics, Oakland University, Rochester, MI 48309-4479}
 \ead{toeniskoetter@oakland.edu}


%\author{}

%\address{}

\begin{abstract}
Let $F$ be a field,  let $D$ be a  
local subring of $F$, and let $\Dom_F(D)$ be the space of valuation rings of $F$ that dominate $D$. We lift Zariski's connectedness theorem for fibers of a projective morphism  to the Zariski-Riemann space of  valuation rings of $F$ by proving that 
 a subring $R$ of $F$ dominating $D$  is local, residually algebraic over $D$ and integrally closed in $F$ if and only if there is a closed and connected subspace $Z$  of $\Dom_F(D)$ such that $R$ is the  intersection of the rings in $Z$.
 Consequently, the intersection of the rings in any closed and connected subset of $\Dom_F(D)$ is a local ring. 
 In proving this, we also prove a converse to 
 Zariski's connectedness theorem. Our results do not  require   the rings involved to be Noetherian. 
%
%\subseteq R$ be   
% subrings of $F$ with $D$ local.  We give a topological criterion for when a subspace of the space of vala
%We prove that
%a subring $R$ of $F$ dominating $D$  is local, residually algebraic over $D$ and integrally closed in $F$ if and only if there is a closed and connected subspace $Z$  of valuation rings of $F$ that dominate $D$  such that $R$ is the  intersection of the rings in $Z$.

\end{abstract}


\end{frontmatter}

\section{Introduction}

In one of its traditional forms, Zariski's Connectedness Theorem asserts that 
if 
$D$ is a Noetherian local
 subring of a field $F$ that is integrally closed in $F$  
and $f:X \rightarrow \Spec(D)$ is a dominant projective morphism of  integral schemes such that $X$ has
 function field contained in  $F$, then  the fiber of $f$ over the closed point of $\Spec(D)$ is closed and connected (cf.~\cite[Corollary~11.3, p.~279]{MR0463157}). 
 That $D$ is local and integrally closed is needed here, and part of the purpose of this article is to examine the relationship between these two properties and the connectedness of the closed fibers of projective morphisms. We do this by fixing the field $F$ and its subring $D$ and varying the projective morphisms $X \rightarrow \Spec(D)$.   
Taking the inverse limit of all such projective morphisms   for which $X$ has function field contained in $F$ results in the Zariski-Riemann space of valuation rings of $F$ that contain $D$. It is in this space that we consider the relationship between local rings and connectedness.

% Thus  if $D$ is a normal Noetherian local  domain, and $f:X \rightarrow \Spec D$ is a blowup of an ideal of $D$, the fiber  $f^{-1}({\mathfrak{m}}_D)$ over the maximal ideal ${\mathfrak{m}}_D$ of $D$ is connected. That $D$ is both local and normal is necessary here for the validity of the theorem, and part of the goal of this article is to understand better this necessity. One way we do this is by fixing a field $F$ containing the local domain $D$ and varying the projective morphisms. Taking the inverse limit of  all projective morphisms $X \rightarrow \Spec D$ such that $X$ has function field contained in $F$ results in the Zariski-Riemann space of valuation rings of $F$ that contain $D$.  It is in this space that we examine the connection between local rings and connectness.  

In more detail,
a {\it valuation ring}  of a field $F$ is a subring $V$ of $F$ such that for all $0 \ne x \in F$, $x \in V$ or $x^{-1} \in V$. 
If $D$ is a subring of $F$, then 
 the {\it Zariski-Riemann space} $\Zar_F(D)$ of $F$ is the space of valuation rings of $F$ that contain $D$ as a subring, where $\Zar_F(D)$ is endowed with the Zariski topology (see Section 2). If also $D$ is local, then $\Dom_F(D)$ is the subspace  of $\Zar_F(D)$ consisting of the valuation rings $V$    that {\it dominate} $D$, i.e.\ the maximal ideal of $V$ contains the maximal ideal of $D$. We show that Zariski's Connectedness Theorem lifts to $\Dom_F(D)$ when $D$ is a local ring that is integrally closed in $F$. Thanks to a recent non-Noetherian version of Stein factorization in The Stacks Project \cite[{Tag~03H2}]{stacks-project}, we do not need to assume $D$ is Noetherian; see Theorem~\ref{first direction}.  
 We also prove a converse to this and give a purely topological criterion for when an intersection of valuation rings is a local ring that is  residually algebraic over $D$. The main purpose of this note is to prove:

\begin{theorem}[Valuative Connectedness Theorem]\label{ValuativeConnectednessTheorem}
Let $F$ be a field, and let $D$ be a local subring of $F$.
A subring $R$ of $F$ dominating $D$ is local, residually algebraic over $D$, and integrally closed in $F$ if and only if there is a closed and connected subspace $Z$ of $\Dom_F(D)$ such that $R$ is the intersection of the rings in $Z$.
\end{theorem}

%Every integrally closed domain is the intersection of the  valuation rings of its quotient field that contain it. In this article we examine the subspaces of $\Zar(F)$ whose valuation rings intersect to a local ring, and we show  
%that  connectedness is a decisive topological feature for determining whether an intersection of valuation rings is local. 

%The issue of when valuation rings intersect to a local ring is closely related to that of the connectedness of fibers of projective  morphisms of schemes, such as arises in . 

Using the fact that the continuous image of a connected set is connected, along with the valuative criterion for properness, one can show that Theorem~\ref{ValuativeConnectednessTheorem} implies Zariski's Connectedness Theorem.
Thus one cannot expect to bypass the difficulty of proving Zariski's theorem by giving a direct proof of Theorem~\ref{ValuativeConnectednessTheorem}.

\begin{notation}
We use the following notation.
\begin{enumerate} 
\item
If $R$ is a local ring, then ${\mathfrak{m}}_R$ denotes its maximal ideal.
\item
If $D$ is a subring of a field $F$, then 
$\Zar_F(D)$ is the set of valuation rings of $F$ that contain $D$.
\item

The set $\Dom_F(D)$ is the set of valuation rings $V$ in $\Zar_F(D)$ such that $\mathfrak m_V \cap D$ is a maximal ideal of $D$, i.e.\ $V$ dominates a localization of $D$ at a maximal ideal.  
%\item
% If $F$ is the quotient field of $D$, we write $\Zar(D)$ for $\Zar_F(D)$ and $\Dom_F(D)$ for $\Dom_F(D)$.  
 \item
If $X$ is a collection of local subrings of $F$, then $A(X)$ is the intersection of these rings, and $J(X)$ is the intersection of the maximal ideals of these rings. 
\item
If $X$ is a projective model of $F/D$  and $Z$ is a subset of $\Zar_F(D)$, then $X(Z)$ denotes the local rings in $X$ that are dominated by a valuation ring in $Z$.
\end{enumerate}

\end{notation}

\section{Projective models and the Zariski-Riemann space}

 In this section, we review the topology of the Zariski-Riemann space of a field and its relationship to projective models of the field.
Rather than use the language of projective schemes, we use the language of projective models to make the interplay between projective schemes and the Zariski-Riemann space of valuation rings easier to work with.

Let $F$ be a field,  let $D$ be a subring of $F$, and let Loc$_F(D)$ be the set of local subrings of $F$ that contain $D$ as a subring and have $F$ as their quotient field. The {\it Zariski topology} on Loc$_F(D)$ has a basis of  open sets the sets of the form 
    $${\mathcal{U}}(x_1,\ldots,x_n):=\{R\in {\rm Loc}(F/D):x_1,\ldots,x_n \in R\},$$
where $x_1,\ldots,x_n\in F$.  
The Zariski topology on Loc$_F(D)$ is quasicompact and the basic open sets  ${\mathcal{U}}(x_1,\ldots,x_n)$ are also quasicompact. In fact, by  \cite[Corollary~2.14]{MR3565806} or \cite[Example~2.1(7)]{MR3565813}, Loc$_F(D)$ is  a {\it spectral space},    a quasicompact $T_0$ topological space such that every open set is a union of  quasicompact open sets; the intersection of two quasicompact open sets is a quasicompact open set; and every irreducible closed subset has a unique generic point.  
(By a theorem of Hochster \cite{MR251026}, spectral spaces are   the topological spaces that are homeomorphic to the prime spectrum of a commutative ring.)  %(For the definition of a spectral space, see the appendix.)

The topology of a spectral space  can be refined to a Hausdorff topology called the {\it patch topology}. This topology has as a basis of open  sets the set of the form ${\mathcal{U}} \cap {\mathcal{V}}$, where ${\mathcal{U}}$ is a quasicompact open set  and ${\mathcal{V}}$ is the complement of a quasicompact open set. The sets ${\mathcal{U}} \cap {\mathcal{V}}$ are both closed and open, and it follows that the patch topology on $X$ is quasicompact, Hausdorff and zero-dimensional. 
%The {\it patch topology} on ${\rm Loc}(F/D)$ is the topology having as basis of open sets the quasicompact open sets  and their complements.
%The Zariski topology on  Loc$_F(D)$ is   Hausdorff only in the very special case in which $D$ is a field and $F$ is algebraic over $D$,  a situation we will not need to consider. However, the patch topology on Loc$_F(D)$ is Hausdorff in all cases and has a basis of clopen sets.
Closure of a set in a spectral space can be viewed as a generic closure of a patch closure. To state this, we recall that the {\it specialization order} on $X$ is given by $x \leq y$ if and only if $y$ is in the closure of $\{x\}$ with respect to the spectral topology on $X$.  
The next lemma is well known; a proof can be found in 
\cite[
%\href{https://stacks.math.columbia.edu/tag/0903}
{Tag 0903}]{stacks-project}.

\begin{lemma}  \label{generic closure} Let $Y$ be patch closed subspace of a spectral space $X$. The closure of a subspace $Y$ of $X$ in the spectral topology is the set of all $x \in X$ such that $y \leq x$ for some $y \in Y$. %
%
%is the set of local rings $R$ in $X$ such that $R$ is contained in a local ring in the patch closure of  $Y$. 
\end{lemma}

If $X$ is a subset of Loc($F/D$), 
 then the map $f:X \rightarrow \Spec  D $ that sends a local ring $R$ in $X$ to the prime ideal ${\mathfrak{m}}_R \cap D$ of $D$ is a \emph{spectral map}, i.e.\ continuous with respect to both the Zariski and patch topologies.
 In particular, if $X = \{D_P:P \in \Spec D\}$, then the map $X \rightarrow \Spec D $ is a homeomorphism, and so the Zariski topology on Loc$_F(D)$  can be viewed as extended from the Zariski topology on $\Spec D$.

Patch closures of subspaces of Loc$_F(D)$ can be intricate to calculate; see for example 
\cite[Section~2]{MR3565813}.  
We work with two classes of subspaces of Loc$_F(D)$, those consisting of valuation rings and those consisting of local rings of points in projective schemes. 
To define the former, 
let  $\Zar_F(D)$ denote the set of valuation rings having quotient field $F$ and containing $D$ as a subring. The {\it Zariski-Riemann space} of $F/D$ is the set $\Zar_F(D)$ endowed with the subspace topology of the Zariski topology on Loc$_F(D)$. With this topology, $\Zar_F(D)$ is a spectral space. This 
%can be seen directly (e.g., as in \cite[Theorem 10.5, p.~74]{MR1011461}), and it 
follows   from the fact that $\Zar_F(D)$ is a patch closed subspace  of Loc$_F(D$); see \cite[Example 2.2(8)]{MR3565813} and \cite{MR0389876}. 

The other class of subspaces of Loc$_F(D)$ that we consider are those arising from  local rings of points in a projective integral scheme over $\Spec D $ with function field contained in $F$. To simplify the presentation, we view such schemes as projective models of $F/D$. Let $x_1,x_2,\ldots,x_n$ be nonzero elements of $ F$, and for each $i \in \{1,2,\ldots,n\}$, let $D_i = D[x_1x_i^{-1},\ldots,x_nx_i^{-1}]$.  Following \cite[p.~119]{MR0389876}, the {\it projective model of $F/D$ defined by $x_1,\ldots,x_n$} is the subspace of Loc$_F(D)$ with respect to the Zariski topology  consisting of all the localizations of the $D_i$, $i \in \{1,2,\ldots,n\}$, at the prime ideals of~$D_i$. 
%When $D$ has quotient field $F$ and $x_1,\ldots,x_n \in D$, then the projective model of $F/D$ defined by $x_1,\ldots,x_n$ is the blow-up Proj$(D[It])$ of $\Spec D$ along the subscheme defined by the ideal $I=(x_1,\ldots,x_n)D$. However, we will consider cases in which $D$ does not have quotient field $F$ and in this case  a projective model is not a blow-up of $\Spec D$. 

Since we will be passing back and forth between projective models and the Zariski-Riemann space, we will typically denote elements of projective models with Greek letters to help distinguish them. 
If $X$ is a  projective model  of $F/D$, then each valuation ring $V$ in $\Zar_F(D)$ {dominates} a unique local ring $\alpha$ in $X$. This property is referred to as {\it irredundance} in \cite[p.~115]{MR0389876}. Also, each local ring $\alpha$ in $X$ is dominated by a valuation ring in $\Zar_F(D)$, a property that is called {\it completeness} in  \cite[p.~127]{MR0389876}. (Irredundance and completeness reflect   the fact that a projective morphism is separated and proper.) 

%Thus there is a surjective mapping $d_X:\Zar_F(D)\rightarrow X$ that sends a valuation ring $V$ in $\Zar_F(D)$ to the local ring in $X$ that  $V$ dominates. 

\begin{lemma} \label{ccs}
{\em \cite[Lemma 4, p.~117]{MR0389876}} For each projective model $X$ of $F/D$, the   mapping $d_X:\Zar_F(D)\rightarrow X$ that sends a valuation ring $V$  in $\Zar_F(D)$ to the unique local ring in $X$ that  $V$ dominates  is surjective  and it is  continuous and closed   with respect to both the Zariski and patch topologies. 
\end{lemma}

The last assertion that $d_X$ is continuous and closed in the patch topology follows from \cite[Theorem~4.1]{MR932760} or \cite[Remark~3.2]{MR3299704}.
We denote by $X(Z)$ the image of $Z$ in $X$ under the mapping $d_X$.  
If $X$ and $Y$ are projective models of $F/D$, then $Y$ {\it dominates} $X$ if each local ring $\alpha$ in $Y$ dominates a local ring $\beta$ in $X$.

A projective model $X$ can be identified with a locally ringed space (in fact, a projective scheme over $\Spec D$) in a natural way by equipping $X$ with a sheaf ${\mathcal{O}}_X$ such that for each nonempty Zariski open subset ${\mathcal{U}}$ of $X$, ${\mathcal{O}}_X({\mathcal{U}})$ is the intersection of the rings in ${\mathcal{U}}$. 
With this observation the next lemma follows from \cite[Exercise~II.3.2, p. 91]{MR0463157} and
\cite[Lemma~5, p.~119]{MR0389876}. 

\begin{lemma}  If $X$ and $Y$ are projective models of $F/D$ and $Y$ dominates $X$, then the map $d_{Y,X}:Y \rightarrow X$ that sends a local ring $\alpha$ in $Y$ to the unique local ring in $X$ that $\alpha$ dominates is surjective, closed and continuous in the Zariski and patch topologies. Moreover, $d_{Y,X} \circ d_Y = d_X$. 
\end{lemma} 

When equipped with the sheaf whose ring of sections over a  nonempty Zariski open subset  is the intersection of the rings in the  set,  $\Zar_F(D)$ is a locally ringed space and is the inverse limit of projective models. The following lemma is implicit in \cite{MR0389876}. For details, see \cite[Proposition~3.3(4)]{MR3299704}.

\begin{lemma} \label{inverse limit}   The set of projective models $Y$ of $F/D$ that dominate $X$ forms an inverse  system when equipped with the maps $d_{Y,X}$. The inverse limit of this system in the category of locally ringed spaces is $\Zar_F(D)$.  
\end{lemma}

\section{Valuative connectedness}

In this section we prove one direction of  the valuative connectedness theorem stated in the introduction by proving that if $R$ is a local ring between the local ring $D$ and the field $F$ and $R$ dominates $D$ and is residually algebraic over $D$, then $\Dom_F(R)$ is a closed connected subset of $\Dom_F(D)$.  
This will be done by passing from $\Dom_F(D)$ to the projective models of $F/D$, where we can use Stein factorization to prove connectedness of fibers. 
For this reason, most of the results in this section deal with passing back and forth from $\Dom_F(D)$ to  relevant projective models of $F/D$. 
But first we make explicit a connection between residually algebraic local rings $R$ dominating $D$ and closed sets of $\Dom_F(D)$. Throughout the paper, working with closed connected sets rather than just connected sets is important for our methods. (This situation can always be arranged by replacing $D$ with a local ring between $D$ and $R$, over which $R$ is residually algebraic.)

\begin{lemma} \label{J lemma} Let $D$ be a subring of the field $F$,  let $Z$ be a nonempty subset of $\Zar_F(D)$, and let $\overline{Z}$ be its Zariski closure in $\Zar_F(D)$.   Then $J(Z) = J(\overline{Z})$. 
\end{lemma}

\begin{proof}
 The patch closure $Z'$ of $Z$ is contained in $\overline{Z}$. Also, $\bigcap_{V \in Z}{\mathfrak{m}}_V = \bigcap_{V \in Z'}{\mathfrak{m}}_V$ by \cite[Lemma 2.3(1)]{MR3686983}. By Lemma~\ref{generic closure}, each valuation ring $V$ in $\overline{Z}$ is contained in a valuation ring $W$ in $Z'$. Since ${\mathfrak{m}}_W \subseteq  {\mathfrak{m}}_V$, it follows that $\bigcap_{V \in Z}{\mathfrak{m}}_V = \bigcap_{V \in \overline{Z}}{\mathfrak{m}}_V$. 
\end{proof}

\begin{proposition} \label{RA lemma} Let $D$ and $R$ be local subrings of
a field $F$ such that $R$  dominates $D$. 
Then $R$ is residually algebraic over $D$ if and only if $\Dom_F(R)$ is closed in $\Dom_F(D)$.
%
%Then $R = A(Z)$, where $Z$ is the Zariski closure of $\Dom_F(R)$ in $\Dom_F(D)$. 
\end{proposition} 

\begin{proof} 
 Suppose $R$ is residually algebraic over $D$. Then $R$ is integral over $D+{\mathfrak{m}}_R$, and so every valuation ring of $F$ containing $D+{\mathfrak{m}}_R$ contains $R$. Since $\Dom_F(R)$ is patch closed in $\Dom_F(D)$, the Zariski closure of $\Dom_F(R)$ in $\Dom_F(D)$ is the set of valuation rings $V \in \Dom_F(D)$ for which $V \subseteq U$ for some $U \in \Dom_F(R)$. Let $V$ be  such a valuation ring. 
% Let $U \in \Dom_F(R)$ and $V \in \Dom_F(D)$ with $V \subseteq U$. %Then $V$ is in the Zariski closure of $\Dom_F(R)$ in $\Dom_F(D)$. 
 By Lemma~\ref{J lemma},
 ${\mathfrak{m}}_R \subseteq J(\Dom_F(R)) \subseteq {\mathfrak{m}}_V$, and  so   
$V$ dominates $R$, which shows $V \in \Dom_F(R)$ and hence $\Dom_F(R)$ is Zariski closed in $\Dom_F(D)$.   

Conversely, if $\Dom_F(R)$ is Zariski closed in $\Dom_F(D)$, there exists a valuation ring $V \in \Dom_F(R)$ that is minimal over $D$, and hence is residually algebraic over $D$. Since $V$ dominates $R$, $R$ is residually algebraic over $D$. 
 \end{proof}

The next lemma will be used to reduce arguments to the case of projective models. 

\begin{lemma} \label{directed models} \label{inverse system}
Let $D$ be a subring of the field $F$,  and let $Z$ be a nonempty subset of $\Zar_F(D)$.   
\begin{enumerate}
\item[$(1)$] The ring $A(Z)$ is the directed union of the rings $A(X(Z)),$ where $X$ ranges over the projective models of $F/D$.
%$$A(Z) =\bigcup_XA(X(Z)).$$

% where  $$A(Z) =\bigcup_XA(X(Z)),$$ 
 
\item[$(2)$] The  ring $A({Z})$ is local if and only if $A(X({Z}))$ is local for each projective model $X$ of $F/D$.

\item[$(3)$] If $Z = \bigcap_{i}Z_i$ for    an inverse system $\{Z_i\}$ of   patch closed subsets of $\Zar_F(D)$, then $\textstyle{A(Z)}$ is the directed union of the $ A(Z_i).$
\end{enumerate}
\end{lemma}

\begin{proof} 
Let $x_1,\ldots,x_n \in A(Z)$, and let $X$ be the projective model defined by the elements $1,x_1,\ldots,x_n$. Then $D[x_1,\ldots,x_n] \subseteq A(Z)$ and the localizations of $D[x_1,\ldots,x_n]$ at prime ideals are in $X$. 
Let $\alpha$ be a local ring in $X(Z)$.  Then $\alpha$ is dominated by a valuation ring $V$ in $Z$.  Since $D[x_1,\ldots,x_n] \subseteq V$, if ${\mathfrak{p}}$ is the center of $V$ in $D[x_1,\ldots,x_n]$, then $D[x_1,\ldots,x_n]_{\mathfrak{p}} \in X(Z)$. By the irredundance property discussed in Section~2, $\alpha = D[x_1,\ldots,x_n]_{\mathfrak{p}}$ since both of these local rings are in $X$ and are dominated by $V$. Thus $x_1,\ldots,x_n \in \alpha$, and since the choice of $\alpha$ in $X(Z)$ was arbitrary,  
  it follows that $x_1,\ldots,x_n \in A(X(Z))$. 
This proves that $A(Z)$ is contained in  the union of the $A(X(Z))$ and this union is directed. To prove the reverse inclusion, 
suppose $x \in A(X(Z))$ for some projective model $X$ of $F/D$. The fact that  $Z \subseteq d_X^{-1}(X(Z))$ implies
$x \in A(d_X^{-1}(X(Z))) \subseteq A(Z)$, and so $A(X(Z)) \subseteq A(Z)$ for all projective models $X$ of $F/D$.

To see that (2) holds, note first that if  $A(X({Z}))$ is local for each projective model $X$ of $F/D$, then $A({Z})$ is a local ring since by (1), $A({Z})$ is the directed union of the rings $A(X({Z}))$, where $X$ ranges over the projective models of $F/D$.
Conversely, suppose $A(Z)$ is local, and let $X$ be a projective model of $F/D$. 
%For each $V \in {\mathcal{V}}$, let $R_V$ denote the ring in $X$ dominated by $V$. 
Let $a$ and $b$ be nonunits in $A(X({Z}))$. Since $A(X(Z))$ is an intersection of the  local rings in $X(Z)$, 
there are local rings $\alpha$ and $\beta$ in $X({Z})$ such that $a \in {\mathfrak{m}}_\alpha$ and $b \in {\mathfrak{m}}_\beta$. Since $\alpha,\beta \in X({Z})$, there are 
 valuation rings $V,W \in {Z}$ such that $V$ dominates $\alpha$ and $W$ dominates $\beta$.  Thus 
 $a \in {\mathfrak{m}}_V$ and $b \in {\mathfrak{m}}_W$, and  $a$ and $b$ are nonunits in $A(Z)$.  This implies  $a+b$ is a nonunit in the local ring $A(Z)$ and hence in the subring $A(X({Z}))$ of $A(Z)$, which shows the sum of nonunits in  $A(X({Z}))$ is itself not a unit. Therefore, $A(X({Z}))$ is local.

For the proof of (3),  suppose $Z = \bigcap_i Z_i.$
 Since $Z \subseteq Z_i$ for each $i$, it is clear that $\bigcup_{i}A(Z_i)
\subseteq  A(Z)$. Suppose $x_1,\ldots,x_n \in A(Z)$, and let 
\begin{center}
${\mathcal{W}} = \{V \in \Zar_F(D):x_i \not \in V$ for some $1 \leq i \leq n\}$. 
\end{center}
Then $Z \cap {\mathcal{W}} = \emptyset$ since $x_1,\ldots,x_n \in A(Z)$, and so $\bigcap_{i} (Z_i \cap {\mathcal{W}}) = \emptyset$. 
Now $\Zar_F(D)$ is compact in the patch topology and  
the sets $Z_i \cap {\mathcal{W}}$ are patch closed subsets of  $\Zar_F(D)$, so there exist $i_1,\ldots,i_n$ such that $ Z_{i_1} \cap \cdots \cap Z_{i_n} \cap {\mathcal{W}} =\emptyset$. Choose $i$ such that  $Z_i \subseteq Z_{i_1} \cap \cdots \cap Z_{i_n}$. Then $Z_i \cap {\mathcal{W}} = \emptyset$, which implies $x_1,\ldots,x_n \in A(Z_i)$, proving that   $A(Z)$ is a directed union of the $A(Z_i).$
\end{proof}

If a  subspace  of a topological space  is connected, then so is its closure in the space. The next lemma shows that the converse is true in $\Zar_F(D)$.

\begin{lemma} \label{dense} Let $D$ be a subring of a field $F$, and let $Z$ be a patch closed subset of $\Zar_F(D)$. Then $Z$ is connected if and only if the Zariski closure of $Z$ in $\Zar_F(D)$ is connected. 
%
%let $Y \subseteq Z$ be patch closed subspaces of $\Zar_F(D)$ such that $Y$ is Zariski dense in $ and let $Y$ be a Zariski dense subspace of $Z$ that is patch closed in $\Zar_F(D)$.  Then
%$Y$ is connected if and only if $Z$ is connected. 
\end{lemma}

\begin{proof}
The closure of a connected subspace of a topological space is always connected, so we need only show the inverse: if $Z$ is not connected, then neither is its Zariski-closure $\overline{Z}$.
Write $Z$ as the union of disjoint nonempty Zariski-closed subspaces $F_1, F_2$ of $Z$.
Since $\overline{Z}$ is the union of the nonempty Zariski-closed sets $\overline{F_1}, \overline{F_2}$, it suffices to show that $\overline{F_1}, \overline{F_2}$ are disjoint.

Suppose not, so there exists $y \in \overline{F_1} \cap \overline{F_2}$.
Since $Z$ is patch-closed set and $F_1, F_2$ are Zariski-closed subspaces, the sets $F_1, F_2$ are also patch-closed.
By Lemma~\ref{generic closure}, there exists $x_1 \in F_1, x_2 \in F_2$ such that $x_1 \le y$ and $x_2 \le y$.
Since the overrings of a valuation ring are totally ordered by inclusion, either $x_1 \le x_2$ or $x_2 \le x_1$, so assume without loss of generality that $x_1 \le x_2$.
But then, since $x_1 \in F_1$ and $F_1$ is a Zariski-closed subspace of $Z$, it follows that $x_2 \in F_1$, contradicting the fact that $F_1, F_2$ are disjoint.
\end{proof}

% Since $\overline{Z}$ is connected, there is $V \in \overline{Z_1} \cap \overline{Z_2}$. Since $Z_1$ and $Z_2$ are Zariski closed in $Z$ and $Z$ is patch closed in $\Zar_F(D)$,  $Z_1$ and $Z_2$ are also patch closed in $\Zar_F(D)$.   By Lemma~\ref{generic closure}, $V \subseteq U_1 \cap U_2$ for some $U_1 \in Z_1$ and $U_2 \in Z_2$. 
%The overrings of a valuation ring are totally ordered under inclusion, so  
%  without loss of generality,
% $V \subseteq U_1 \subseteq U_2$. Since $U_1 \in Z$, $U_2 \in Z_2$ and $Z_2$ is Zariski closed in $Z$, we have $U_1 \in Z_2$. Thus $U_1 \in Z_1 \cap Z_2$, proving that $Z$ is not the disjoint union of nonempty Zariski closed subsets and hence that $Z$ is   connected.   

%  %The next  lemma is used in Section~3 to reduce the question of connectedness of a space of valuation rings to that of connectedness of fibers in  projective models. 

In general, an inverse limit of connected sets need not be connected. However, connectedness does lift to the inverse limit under the  hypotheses that we work with in this article. We give the topological argument for this in the next lemma, which  is surely known but for which we could not find a reference.
 
\begin{lemma} \label{spectral connected 1}
Let $S$ be the inverse limit of topological spaces $X_i$ such that the projection maps $d_i : S \rightarrow X_i$ are closed maps, and let $Z \subseteq S$ be a nonempty closed quasicompact subset.
Then $Z$ is connected if and only if $d_i (Z)$ is connected for each~$i$. 
\end{lemma}
 
\begin{proof}
Since the continuous image of a connected set is connected, if $Z$ is connected, then so is $d_i (Z)$ for each $i$.
Conversely, assume that $d_i ({Z})$ is connected for each $i$.
To show $Z$ is connected, for any pair of nonempty closed subsets $F_1, F_2 \subseteq Z$ such that $Z = F_1 \cup F_2$, we claim that $F_1 \cap F_2 \ne \emptyset$.
For each $i$, since $d_i$ is a closed map, we have that $d_i (F_1), d_i (F_2)$ are closed sets in $X_i$.
Since $d_i ({Z})$ is connected and $d_i ({Z}) = d_i (F_1) \cup d_i (F_2)$, the intersection $d_i (F_1) \cap d_i (F_2)$ must be nonempty.
For each $i$, define $Z_i$ to be the preimage of this intersection in $Z$, that is
    $$Z_i = Z \cap  d_i^{-1} (d_i (F_1) \cap d_i (F_2)).$$
Notice that for any $i, j$, if $i \ge j$, then $Z_i \subseteq Z_j$.
To see this, let $d_{j, i} : X_i \rightarrow X_j$ denote the compatibility map, where $d_{j, i} \circ d_i = d_j$.
Since $d_i (Z_i) \subseteq d_i (F_1) \cap d_i (F_2)$, by applying the map $d_{j, i}$, we obtain $d_j (Z_i) \subseteq d_j (F_1) \cap d_j (F_2)$.
Thus,
    $$Z_i \subseteq Z \cap d^{-1}_j (d_j (F_1) \cap d_j (F_2)) = Z_j.$$
This implies that the collection $\{ Z_i \}$ has the finite intersection property.
Indeed, let $j_1, \ldots, j_n$ be in the index set, and choose $j$ with $j \ge j_k$ for each $k$, $1 \le n$.
Then $Z_j$, which is nonempty, is a subset of each $Z_{j_k}$, and thus a subset of their intersection.

Since $\{ Z_i \}$ is a set of closed subsets of $Z$ with the finite intersection property and $Z$ is quasicompact, the intersection $\bigcap_{i} Z_i$ is nonempty. 
From the fact that $F_1, F_2$ are closed it follows\footnote{The general fact being used here is:
{\it
    Let $S$ be a topological space  that is an inverse limit of topological spaces $X_i$, and for each $i$ let $d_i: S \rightarrow X_i$ denote the projection map.
    If $Z$ is a closed subset of $S$, then $Z = \bigcap_{i} d_i^{-1}(d_i(Z))$.
}
To see that this is the case, observe that since the sets of the form $d_i^{-1} (Y)$, where $Y$ is a closed set in $X_i$, comprise a basis for the closed sets of $S$, it follows that $Z$ is the intersection of sets of the form $d_i^{-1}(Y)$, where $d_i(Z) \subseteq Y$ and $Y$ is closed in $X_i$. 
For such a set $Y$ we have $Z \subseteq d^{-1}_i(d_i(Z)) \subseteq d_i^{-1}(Y)$, and so $Z$ is the intersection of the sets of the form $d_i^{-1}(d_i(Z))$. 
 }
that $F_k =\bigcap_{i} d_i^{-1} (d_i (F_k))$ for $k = 1, 2$.
Thus 
\begin{eqnarray*} 
    \bigcap_i Z_i 
        &= & Z \cap \left( \bigcap_i d_i^{-1}(d_i(Z_1) \cap d_i(Z_2))\right) \\
        &= & Z \cap \left( \bigcap_i  d_i^{-1}(d_i(Z_1))\right) \cap \left( \bigcap_i d_i^{-1}( d_i(Z_2))\right) \\ 
        &= & Z_1 \cap Z_2,
\end{eqnarray*}
proving that $Z_1 \cap Z_2 \ne \emptyset$.
\end{proof}
 
%The lemma gives a way to  relate connectedness in $\Zar_F(D)$ with connectedness in the projective models of $F/D$:
 
\begin{lemma} \label{connected iff} Let $D$ be a subring of the field $F$.  A nonempty patch closed  subset ${Z}$ of $\Zar_F(D)$ is Zariski connected   if and only if $X({Z})$ is Zariski connected  for each projective model $X$ of $F/D$.  
\end{lemma} 

\begin{proof} 
If $Z$ is connected, then as a continuous image of $Z$, $X(Z)$ is  connected. Conversely, suppose $X(Z)$ is connected for all projective models $X$ of $F/D$.  Since the domination map $d_X$ is closed, the Zariski closure of $X(Z)$ in $X$ is $X(\overline{Z})$, where $\overline{Z}$ is the Zariski closure of $Z$ in $\Zar_F(D)$. Since a space containing a dense connected subspace is connected, $X(\overline{Z})$ is connected for all projective models $X$ of $F/D$.  
By Lemma~\ref{inverse limit}, the set of projective models of $F/D$ is an inverse system with respect to the domination maps for these models, and the inverse limit of this system is $\Zar_F(D)$. The domination maps $d_i:\Zar_F(D) \rightarrow X_i$ are the projection maps for the inverse limit, and by Lemma~\ref{ccs} these continuous maps are closed. Thus Lemma~\ref{spectral connected 1} implies $\overline{Z}$ is connected. By Lemma~\ref{dense}, $Z$ is connected. 
 \end{proof}

Using the previous lemmas, we now     prove that local rings give rise to connected sets by reducing to   the case of projective models, where Stein factorization is available. 

\begin{theorem} \label{first direction} Suppose $D \subseteq R$ are local subrings of a field $F$ such that $R$ dominates $D$, $R$ is integrally closed in $F$ and $R$ is residually algebraic over $D$. Then   $\Dom_F(R)$ is a closed connected subset   of $\Dom_F(D)$.  
\end{theorem}

\begin{proof}

Let $Z=\Dom_F(R)$. To show $Z$ is connected, it suffices by  Lemma~\ref{connected iff} to show $X(Z )$ is connected for each projective model $X$ of $F/R$. Let $X$ be a projective model of $F/R$. 
As follows from the property of completeness discussed in Section~2 (or from the valuative criterion for properness), the image $X(Z)$ of $Z$ in $X$ under the domination map $d_X:\Zar_F(D) \rightarrow X$ is the fiber of
 the map $X \rightarrow \Spec R$ over the closed point 
  of $\Spec R$. Since  $R$ is integrally  closed in $F$, 
$R$ is integrally closed in the   
   function field of $X$. 
Now we apply the non-Noetherian version of Stein factorization from \cite[{Tag~03H2}]{stacks-project}, which in this case implies that since $R$ is  local and integrally closed in $F$, the fiber $X(Z)$ of the map
  $X \rightarrow \Spec R$ over the closed point 
  of $\Spec R$ is
     connected in the Zariski topology.
%  Since $d_X$ is a closed map,   the image $X( {Z})$ of $Z$ under $d_X$ is the Zariski closure of  $X(Z)$ in $X$.  
This proves
    ${Z}$ is connected. That $Z$ is closed follows from Proposition~\ref{RA lemma}. 
%Apply Lemmas~\ref{Noetherian local} and~\ref{RA lemma} and use the fact that a set is connected if it has a dense connected subset. 
\end{proof}

\section{Local intersections over closed and connected sets} 

The purpose of this section is to prove a   converse to  Zariski's  connectedness theorem (Theorem~\ref{proj case}) and to lift this converse to spaces of valuation rings (Theorem~\ref{is local}).  

%If $X$ is a projective model over $D$, then the {\it exceptional locus} of $X$ is the (closed) set of local rings in $X$ that dominate $D$. 

\begin{theorem} \label{proj case} 
Let $D$ be a local subring of a field $F$ such that $\Spec D$ is a Noetherian space, and let $X$ be a projective model of $F/D$.
\begin {enumerate}
\item[$(1)$]
If $Y$ is a Zariski closed subset of the set of local rings in $X$ that dominate $D$, then $A(Y)$ is a semilocal ring whose maximal ideals are contracted from the maximal ideals of the rings in $Y$.  
\item[$(2)$]
If in addition  $Y$ is connected,  then $A({Y})$ is a   local   ring that is dominated by each ring in $Y$. %If also the model $X$ is normal, then $A$ is a normal Noetherian ring. 
\end{enumerate}
\end{theorem}

\begin{proof} 
Since $\Spec D$ is Noetherian, each finitely generated $D$-algebra has  Noetherian prime spectrum \cite[Section~2]{MR229627}, and it follows that $X$, as a finite union of Noetherian spaces, is a Noetherian space. Thus $Y$, as a closed subset of $X$, is the union of finitely many closed irreducible subsets ${Y}_1,{Y}_2,\ldots,{Y}_n$ of $X$. %If ${Y}$ consists of a single local ring, then clearly $A({Y})$ is a local ring, so we assume ${Y}$ has more than one member. 
%Since ${Y}$ is connected, each irreducible component of ${Y}$ has infinitely many members, and so, as a closed set, ${Y}$ contains local rings that are not closed points in $X$. 
For each $i$, let $U_i$ be the local ring that is the unique generic point of ${Y}_i$.  Then $U_i$ is 
essentially of finite type over $D$, and  $U_i$  dominates $D$. Also, ${Y}_i = \{\alpha \in {Y}:\alpha \subseteq U_i\}$. 
%
%and are residually transcendental over $D$.  
The Zariski closed set ${Y}$ thus consists of the rings in $X$ that are contained in at least one of the $U_i$. 

%Since $X$ is a nonsingular projective model over $D$ and the local rings in ${Y}$ dominate $D$, there are at most two one-dimensional  local rings in ${Y}$ {\bf [ref]}. It follows that ${Y}$ is either the exceptional fiber of $X$ over $D$ or ${Y}$ is an irreducible component whose generic point is a prime divisor of $D$. 

%In the former case, $A({Y})$ is the intersection of all the two-dimensional regular local rings in $X$ and hence is $D$ itself. Thus we need only consider the latter case in which ${Y}$ is an irreducible component whose generic point is a prime divisor $U$ of $D$.
 
 Let $A = A({Y})$.
Then $A \subseteq U_1 \cap \cdots \cap U_n$. For each $i \in \{1,2,\ldots,n\}$,  let $P_i = {\mathfrak{m}}_{U_i} \cap A$. We claim each $P_i$ is a maximal ideal of $A$. 
%and 
%let ${Y}_i = \{\alpha \in {Y}:\alpha \subseteq U_i\}$ be the irreducible component of ${Y}$ whose generic point is $U_i$. 
Observe that for each $i$, $$
D / \mathfrak{m}_D =
(D+{\mathfrak{m}}_{U_i})/{\mathfrak{m}}_{U_i}
\subseteq 
(A+{\mathfrak{m}}_{U_i})/{\mathfrak{m}}_{U_i}
\subseteq \bigcap_{\alpha \in {Y}_i}(\alpha+ {\mathfrak{m}}_{U_i})/{\mathfrak{m}}_{U_i} \subseteq U_i/{\mathfrak{m}}_{U_i}.$$
Since $U_i$ dominates $D$, the maximal ideal ${\mathfrak{m}}_{U_i}$ of $U_i$ contracts in $D$ to the maximal ideal of $D$. 
By \cite[(6.2.19),p.~161]{MR1617523}, the collection $\{(\alpha+{\mathfrak{m}}_{U_i})/{\mathfrak{m}}_{U_i} \mid \alpha \in {Y}_i\}$ is a projective model over the residue field of $D$. Since the intersection of local rings in a projective model is contained in the intersection of all the valuation rings of the function field of the model, the extension     $$(D+{\mathfrak{m}}_{U_i})/{\mathfrak{m}}_{U_i} \subseteq
\bigcap_{\alpha \in {Y}_i}(\alpha+ {\mathfrak{m}}_{U_i})/{\mathfrak{m}}_{U_i}$$ 
is an integral extension whose base ring is a field.
This implies the intermediate ring $(A+{\mathfrak{m}}_{U_i})/{\mathfrak{m}}_{U_i}$ 
  is a field. Since $P_i = A \cap {\mathfrak{m}}_{U_i}$, we conclude $P_i$ is a maximal ideal of~$A$.

%There is $i \in \{1,2,\ldots,n\}$ such that $\alpha \subseteq U_i$, so ${\mathfrak{m}}_{U_i} \cap \alpha \subseteq {\mathfrak{m}}_{\alpha}$. Thus $P_i = {\mathfrak{m}}_{U_i} \cap A  \subseteq {\mathfrak{m}}_{\alpha} \cap A$. Since $P_i$ is maximal in $A$, we have $x,y \in {\mathfrak{m}}_\alpha \cap A = P_i$, contrary to the chose

Next we claim that the $P_i$ are all the maximal ideals of $A$. 
Let $x \in A$ be a nonunit. Since $A = A (Y)$, there exists $\alpha \in Y$ such that $x \in \mathfrak{m}_{\alpha}$.
Since $\alpha \subseteq U_i$ for some $i$ and  $\alpha \cap \mathfrak{m}_{U_i} \subseteq \mathfrak{m}_{\alpha}$,
it follows that $ P_i = A \cap \mathfrak{m}_{U_i} \subseteq A \cap \mathfrak{m}_{\alpha}$.
Since $P_i$ is a maximal ideal of $A$, we conclude that  $P_i = A \cap \mathfrak{m}_{\alpha}$, and so $x \in P_i$. This proves that every nonunit in $A$ is contained in one of the $P_i$. Using prime avoidance, we conclude that each proper ideal in $A$ is contained in one of the $P_i$, and hence the $P_i$ are all the maximal ideals of $A$.  This proves item 1.

  %Next we claim $P_1$ is the unique maximal ideal of $A$. Suppose by way of contradiction that  there is $p \in P_1$ such that $1+p$ is not a unit in $A$. Then since $A = A({Y})$ there is $\alpha \in {Y}$ such that $1+p$ is in ${\mathfrak{m}}_\alpha$. Since $\alpha \subseteq U_i$ for some $i$, we have $\alpha \cap  {\mathfrak{m}}_{U_i} \subseteq {\mathfrak{m}}_\alpha$, and hence $P_1=P_i = A \cap {\mathfrak{m}}_{U_i} \subseteq A \cap {\mathfrak{m}}_\alpha.$ Since $P_1$ is a maximal ideal of $A$, it follows that $1+p \in A \cap {\mathfrak{m}}_\alpha = P_1$, a contradiction to the assumption that $p \in P_1$. Therefore, $P_1$ is the unique maximal ideal of $A$.  
  %The ring $A$ is residually algebraic over $D$ since it is dominated by local rings that are residually algebraic over $D$.  It is  two-dimensional since it is dominated by two-dimensional local rings that have prime ideals that do not dominate $D$. 
  
For item 2,   suppose $Y$ is connected. To prove that $A(Y)$ is local, we show that $P_1 = P_2 = \cdots = P_n$.
Suppose this is not the case. Since the $P_i$ are maximal ideals we may assume without loss of generality that $P_1 \ne P_2 \cap \cdots \cap P_n$.
Since each $P_i$ is maximal, there exist nonzero elements $x \in P_1$ and $y \in P_2 \cap \cdots \cap P_n$ such that $x + y = 1$.
%Since $xy \in P_1 \cap \cdots P_n$, 
For an element $t\in F$, let $${\mathcal{V}}(t) =\{ \alpha \in X \mid t \notin \alpha \},$$ and note that ${\mathcal{V}}(t)$ is a Zariski closed subset of $X$.
Now $U_1 \in {\mathcal{V}}(x^{-1})$ and $U_2,\ldots,U_n \in {\mathcal{V}}(y^{-1})$, so 
since $U_1,\ldots,U_n$ are the generic points for the irreducible closed components of ${Y}$, we have  ${Y} \subseteq {\mathcal{V}}(x^{-1}) \cup {{\mathcal{V}}}(y^{-1})$. 
%Since $x+y=1$, we have that $x \not \in P_2 \cap \cdots \cap P_n$ and $y \not \in P_1$.  Without loss of generality, assume $x \not \in P_2$. Then $x$ is a unit in $U_2$ but not $U_1$ and $y$ is a unit in $U_1$.  
To see that $Y \cap \mathcal{V} (x^{-1}) \cap \mathcal{V} (y^{-1}) = \emptyset$, suppose $\alpha \in {Y} \cap  {\mathcal{V}}(x^{-1}) \cap {\mathcal{V}}(y^{-1})$.
Since $x, y \in A \subseteq \alpha$ and $x^{-1}, y^{-1} \notin \alpha$, both $x$ and $y$ are nonunits in $\alpha$ with $x + y = 1$, contradicting the locality of~$\alpha$.
From this we conclude ${Y}$ is the disjoint union of the nonempty closed subsets $ {\mathcal{V}}(x^{-1}) \cap {Y}$ and ${\mathcal{V}}(y^{-1}) \cap {Y}$, a contradiction to the fact that  ${Y}$ is connected. Therefore, $P_1 = P_2 = \cdots = P_n$, which proves $A$ is local. Moreover, $A$ is dominated by each $\alpha \in {Y}$ since the previous argument shows  $P_1  = A 
\cap {\mathfrak{m}}_\alpha$.
\end{proof}

We illustrate Theorem~\ref{proj case} with an example from an interesting  geometric setting. The details for the example can be found in \cite[Example~7.6]{MR4097907}.

\begin{example} Let $D$ be a two-dimensional regular local ring with quotient field $F$, and let ${\mathfrak m} =(x,y)D$ be the maximal ideal of $D$.  The ring  $R = D[y^2/x]_{(x,y,y^2/x)D[y^2/x]}$ is a two-dimensional Noetherian normal local domain that is not regular, since its maximal ideal ${\mathfrak{m}}_R=(x,y,y^2/x)R$ requires three generators. The blow-up  Proj $R[{\mathfrak m}_Rt]$ of ${\mathfrak m}_R$ is the desingularization of $R$, which we may view as a  projective model of $F/R$. It is a subset of the  projective model $X= $ Proj $D[Jt]$, where $J = (x^2, xy, y^3)D$. Let $Y$ be the set of local rings in $X$ that dominate $R$. Then $Y$ is a closed subset of $X$, and it is connected since it is irreducible. 
Moreover, $R = A(Y)$, as in 
 Theorem~\ref{proj case}.
\end{example}

More generally, if $D$ is a local Noetherian domain and $J$ is an ideal with blowup  $X = $ {Proj~$D[Jt]$}, then for each closed connected subset $Y$ of the fiber in $X$ over the maximal ideal of $D$, the local rings in $Y$ 
 intersect to a local ring, according to Theorem~\ref{proj case}. In particular, the irreducible closed sets in $Y$ yield local rings.

We   lift  Theorem~\ref{proj case} to the Zariski-Riemann space of valuation rings by using the fact from Lemma~\ref{inverse limit} that this locally ringed space is the inverse limit   of projective models.

\begin{theorem} \label{is local}
Let $D$ be a local subring of the field $F$.
If ${Z}$ is a nonempty Zariski closed and connected subspace of  $\Dom_F(D)$, then $A({Z})$ is a local domain that dominates $D$ and is residually algebraic over $D$. 
\end{theorem}

\begin{proof}
Let $D'$ be the prime subring of $D$. 
We may write $D $ as a directed union $D= \bigcup_{i} D_i$, where for each $i$, $D_i$ is a local ring essentially of finite type over $D'$ and dominated by $D$.  
Then $\Dom_F(D) = \bigcap_{i} \Dom_F(D_i)$.
 For each $i$, let 
 $Z_i$ be the Zariski closure of ${Z}$ in  $\Dom_F(D_i)$. Then $Z_i \cap \Dom_F(D) = Z$ since $Z$ is Zariski closed in $\Dom_F(D)$, so 
 %By Lemma~\ref{closed}, each valuation ring in $Z_i$ dominates $D_i$.
  $Z = \bigcap_i Z_i$. Since $\{Z_i\}$ is an inverse system of Zariski closed (hence patch closed) subsets of $\Zar_F(D)$,
  Lemma~\ref{inverse system}(3) implies $A(Z)$ is a directed union of the $A(Z_i)$. Thus by Lemma~\ref{directed models}(1), $A(Z)$ is the directed union of the rings $A(X(Z_i))$, where $X$ ranges over the  projective models of $F/D_i$ and $i$ ranges over the index set of $\{D_i\}$. 
  
  To prove that $A(Z)$ is local it suffices to prove that $A({X}(Z_i))$ is local for each~$i$ and  each projective model $X$.  
  Let $i$ be in the index set, and let $X$ be a projective model of $F/D_i$. 
 Since $Z_i$ contains the dense connected set $Z$, $Z_i$ is connected. 
 For each   projective model $X$ of $F/D_i$, the map $d_X:\Zar_F(D_i) \rightarrow X$  that sends a valuation ring to the unique local ring it dominates in $X$ is by Lemma~\ref{ccs}  continuous, closed and surjective.
 Thus, since $Z_i$ is    closed and connected, so is the subset $d_{X}(Z_i)=X(Z_i)$ of the set of local rings in $X$ dominating $D_i$,
 and hence by Theorem~\ref{proj case},  $A(X(Z_i))$ is a local ring. Therefore, $A(Z)$ is a local ring.  

 Finally, since each valuation ring in $Z$ dominates $D$, we have ${\mathfrak{m}}_D \subseteq J(Z) \subseteq {\mathfrak{m}}_{A(Z)}$, and so $A(Z)$  dominates $D$.  
 Let $U$ be a valuation ring in $Z$. Since $Z$ is closed in $\Dom_F(D)$, there is a valuation ring $V \subseteq U$ such that $V \in Z$, $V$ is minimal over $D$ and hence $V$ is residually algebraic over $D$.  Thus since $D +{\mathfrak{m}}_V \subseteq A(Z)+{\mathfrak{m}}_V  \subseteq V$, it follows that $A(Z)$ is residually algebraic over $D$. 
\end{proof}

Combining Theorems~\ref{first direction} and~\ref{is local}, we obtain the valuative connectedness theorem from the introduction.

 Theorem~\ref{is local}, and hence the valuative connectedness theorem,  are not true if the assumption that $Z$ is Zariski closed is removed.  Example~\ref{Nagata example} shows that even  patch closed in place of Zariski closed  is not strong enough for the theorem to remain valid.

\begin{example} \label{Nagata example} Let $D$ be a Noetherian normal local domain with Krull dimension at least $2$ and quotient field $F$. Let $x$ be an indeterminate for $F$, and for each subring $R$ of $F$, let $R(x)$ denote the Nagata function ring of $R$, i.e.\ $R(x)$ is the set of rational functions $f(x)/g(x)$ in $F(x)$, where $f(x),g(x) \in R[x]$ and the ideal generated by the coefficients of $g$ is $R$. Let $Z = \Dom_F(D)$. By Theorem~\ref{first direction}, $Z$ is connected in the Zariski topology.  Let $Z^* = \{V(x):V \in Z\}$. Then $Z^*$ is homeomorphic to $Z$ (cf.~\cite[Proposition~4.2(b)]{MR3299704}) and patch closed in $\Dom_{F(x)}(D(x))$ \cite[Proposition~5.6(6)]{MR3299704}, yet $A(Z^*)$ is a B\'ezout  domain whose maximal ideals are the centers of the minimal valuation rings in $Z^*$ \cite[Corollary~5.8]{MR3299704}. 
Since $D$ is Noetherian and  has dimension at least $2$, there are infinitely many minimal valuation rings in $Z$, hence in $Z^*$, and so $A(Z^*)$ has infinitely many maximal ideals.  
Thus $Z^*$ is a patch closed, Zariski connected subset of $\Dom_{F(x)}(D(x))$ for which $A(Z^*)$ is not local. 
\end{example} 

As the example shows,  Theorem~\ref{is local} can fail if $Z$ is not Zariski closed. However,  the theorem can still be applied  
to the Zariski closure of $Z$ and information about $A(Z)$  obtained.

 \begin{corollary} \label{is local corollary} Let $D$ be a local subring of the field $F$ such that $D$ is not a field. If $Z$ is a nonempty connected subspace of $\Dom_F(D)$, then
 there is a local integrally closed subring of $A(Z)$ whose complete integral closure contains $A(Z)$.  

 \end{corollary} 
 
 \begin{proof} 
 Let $\overline{Z}$ denote the Zariski closure of $Z$ in $\Dom_F(D)$.
 Then $\overline{Z}$ is connected since it has a dense connected subspace, and so $A(\overline{Z})$ is by Theorem~\ref{is local} an integrally closed local domain.
 By Lemma~\ref{J lemma}, $J(Z) = J(\overline{Z})$. 
 Since the valuation rings in $Z$ dominate $D$ and $D$ is not a field, there is $0 \ne x \in \bigcap_{V \in Z}{\mathfrak{m}}_V$. Consequently, $xA(Z) \subseteq A(\overline{Z})$, so that $A(Z)$ is in the complete integral closure of $A(\overline{Z})$. \end{proof}

\printbibliography

\end{document}